\documentclass{amsart}
\usepackage{ifthen,amsthm,amssymb,latexsym,amsmath}
\usepackage{mathrsfs}
\usepackage{graphicx}
\input xy
\xyoption{all}
\usepackage[all]{xy}

\def\C{\mathbb C}
\newcommand{\CC}{\mbox{${\mathbb C}$}}

\newcommand{\cod}{\text{\rm cod}}

\renewcommand{\dim}{\mathrm{dim}}

\newcommand{\E}{\mathcal{P}_{\mathscr F}}

\newcommand{\F}{\mathcal{T}_{\mathscr F}}
\newcommand{\fol}{\mbox{${\mathscr F}$}}

\newcommand{\G}{\mbox{${\mathscr G}$}}

\newcommand{\I}{\mbox{${\mathcal I}$}}

\newcommand{\lra}{\longrightarrow}

\newcommand{\nf}{\mathcal{N}_{\mathscr F}}

\renewcommand{\O}{{\mathcal O}}
\newcommand{\opn}{{\mathcal O}_{\mathbb{P}^n}}
\newcommand{\Oc}{\mbox{${\mathcal O}$}}

\newcommand{\pn}{\mathbb{P}^n}
\newcommand{\PP}{\mathbb{P}}

\newcommand{\rk}{\mathrm{rank}}

\newcommand{\singf}{{\rm Sing}({\mathscr F})}

\newcommand{\tpn}{\mathcal{T}_{\mathbb{P}^n}}
\newcommand{\tx}{\mathcal{T}_X}

\newtheorem{lema}{Lemma}[section]
\newtheorem{cor}[lema]{Corollary}
\newtheorem{teo}[lema]{Theorem}

\newtheorem{prop}[lema]{Proposition}
\newtheorem{mthm}{Theorem}
\newtheorem{mcor}[mthm]{Corollary}

\theoremstyle{definition}

\newtheorem{remark}[lema]{Remark}
\newtheorem{defi}[lema]{Definition}
\newtheorem{exe}[lema]{Example}

\begin{document}

\title{On the Singular Scheme of Split Foliations}

%\date{ }are

\begin{abstract}
We prove that the tangent sheaf of a codimension one locally free
distribution splits as a sum of line bundles if and only if its
singular scheme is arithmetically Cohen-Macaulay. In addition, we
show that a foliation by curves is given by an intersection of
generically transversal holomorphic distributions of codimension one
if and only if its singular scheme is arithmetically Buchsbaum.
Finally, we establish that these foliations are determined by their
singular schemes, and deduce that the Hilbert scheme of certain
arithmetically Buchsbaum schemes of codimension $2$ is birational to
a Grassmannian.
\end{abstract}

\author{Maur\'icio Corr\^ea Jr}
\address{Maur\'icio Corr\^ea Jr\\ ICEx - UFMG \\
Departamento de Matem\'atica \\
Av. Ant\^onio Carlos 6627 \\
30123-970 Belo Horizonte MG, Brazil } \email{mauricio@mat.ufmg.br}

\author{Marcos Jardim }
\address{Marcos Jardim\\ IMECC - UNICAMP // Rua S\'ergio Buarque de Holanda, 651, Campinas, SP, Brazil, CEP 13083-859}
\email{jardim@ime.unicamp.br}

\author{Renato Vidal Martins}
\address{Renato Vidal Martins \\  ICEx - UFMG \\
Departamento de Matem\'atica \\
Av. Ant\^onio Carlos 6627 \\
30123-970 Belo Horizonte MG, Brazil } \email{renato@mat.ufmg.br }

\thanks{ }



%\dedicatory{}
%\commby{ }

%\begin{abstract}
%\end{abstract}

%\begin{center}
\maketitle
%\end{center}

%%%%%%%%%%%%%%%%%%%%%%%%%%%%%%%%%%%

\section{Introduction}

The goal of this paper is to characterize holomorphic distributions
on projective spaces either whose tangent sheaf splits as a sum of
line bundles, or which are globally defined as an intersection of
generically transversal distributions of codimension one. The last
case is equivalent to the splitting as a sum of line bundles of a
certain reflexive sheaf canonically defined from the foliation,
which we call \emph{Pfaff sheaf}.

Concerning split tangent sheaves, we start by recalling J. P.
Jouanolou's celebrated 1979 paper \cite{J}. It is doubtless a
landmark in the study of holomorphic foliations in projective
spaces. By means of algebro-geometric tools, he described, for
instance, the irreducible components of the space of
one-codimensional foliations of degree $0$ and $1$. New components
were found later on by Omegar \cite{Omegar} and D. Cerveau and A.
Lins Neto \cite{Cerveau-Lins-Neto}.

In such a study, it turned out to be helpful deciding when the
tangent sheaf of the foliation splits as a sum of line bundles. In
fact, F. Cukierman and J. V. Pereira proved in
\cite{Cukierman-Pereira} that the splitting conditon, along with
some properties for the singular locus, provide stability under
deformations and make it possible to characterize certain components
of the space of foliations.

More recently, L. Giraldo and A. J. Pan-Collantes showed in
\cite{GP} that the tangent sheaf of a foliation of dimension $2$ on
$\mathbb{P}^3$ splits if and only if its singular scheme is
arithmetically Cohen Macaulay (ACM). The first part of the present
article is devoted precisely to extend this result. More precisely,
we prove the following result.

\begin{mthm}\label{2}
Let $\fol$ be a distribution on $\PP^n$ of codimension $k$, such
that the tangent sheaf $\F$ is locally free and whose singular locus
has the expected dimension $n-k-1$. If $\F$ splits as a sum of line
bundles, then $\singf$ is arithmetically Cohen--Macaulay.
Conversely, if $k=1$ and $\singf$ is arithmetically Cohen--Macaulay,
then $\F$ splits as a sum of line bundles.
\end{mthm}

Note that \cite[Thm 3.3]{GP} corresponds to the case $n=3$ and
$k=1$. We emphasize that when $n$ is odd and greater than $3$, a
technical difficulty arises, which we resolve via a spectral
sequence.

The \emph{Pffaf sheaf} of a distribution $\fol$ is defined to be the
dual of its normal sheaf $\nf$. The second part of this paper
concerns distributions whose Pfaff sheaves split as a sum of line
bundles. For distributions with pure dimensional singular locus,
this property is equivalent to saying that the distribution is an
intersection of generically transversal holomorphic distributions of
codimension one.

We say that $\fol$ is a \emph{complete intersection distribution} if
there exist generically transversal
 holomorphic distributions $\fol_1,\ldots,\fol_k$ of codimension one on $\pn$ such that $\fol=\fol_1\cap\dots\cap \fol_k$ is
 a distribution of codimension $k$ on $\pn$. In this case, the Pfaff sheaf of $\fol$ is a locally free sheaf of the form
$$
\E=\bigoplus_{i=1}^k\mathcal{O}_{\pn}(-2-d_i)
$$
where $d_i=\deg(\fol_i)\geq 0$. Then $\fol$ is induced by a $k$-form
$$
\omega \in H^0(\mathbb P^n,\Omega^k_{\mathbb P^n} \otimes
\det(\E^*))
$$
globally decomposable and $\deg(\fol)=d_1+\ldots +d_k+k-1$. In
particular,  $\deg(\fol)\geq k-1.$

Our strategy is to link this subject with arithmetically Buchsbaum
schemes, using a
 characterization of codimension two arithmetically Buchsbaum schemes due to M. C. Chang \cite{Ch}.
 The results we obtained are stated below.

\begin{mthm}\label{3}
Let $\fol$ be a holomorphic  distribution of dimension $r$ on $\pn$.
Suppose that
 $\cod(\singf)=r+1$ and that the induced  Pfaff system  $\E
\rightarrow \Omega_{\pn}^1$ is locally free. Then the following
hold:
\begin{itemize}
\item[(i)] if $r=1$ and $\singf$ is of pure dimension, then the following are equivalent:
\begin{itemize}
\item[$\bullet$]
$\E=\bigoplus_{i=1}^{n-1}\mathcal{O}_{\pn}(-2-d_i)$, with $d_i\geq
0$ for all $i=1,\ldots, n-1.$
\item[$\bullet$] $\singf$ is arithmetically Buchsbaum; $h^1(\mathcal{I}_Z(d-1))=1$; and this is
the only nonzero intermediate cohomology for $H^i_*(\mathcal{I}_Z)$
in the range $ 1\leq i\leq  n-2.$
\item[$\bullet$] $\fol$ is a complete intersection foliation.
\end{itemize}
Moreover, if $\singf$ is smooth and $d_i\geq n$ for some $i$, then
$\singf$ is of general type.
\item[(ii)] if $r=1$, and $\fol'$ is a foliation by curves on $\pn$,
with the same degree of $\fol$ such that $\mathrm{Sing}(\fol)\subset
\mathrm{Sing}(\fol')$, then $\fol'=\fol$.
\item[(iii)] if $r=2$ and $\E$ splits as a sum of line bundles, then $\singf$ is arithmetically Buchsbaum, but not arithmetically Cohen Macaulay;
\item[(iv)] if $r=3$, $\E$ splits as a sum of line bundles, the $|d_i-d_j| \neq 1$, and, for $n\geq 7$, the $d_i\neq 1$
as well, then $\singf$ is arithmetically Buchsbaum, but not
arithmetically Cohen Macaulay.
\end{itemize}
\end{mthm}

An interesting consequence is the construction of holomorphic
foliation  with  prescribed singular scheme  being arithmetically
Buchsbaum.

\begin{mcor}\label{existencia fol}
Let $Z$ be an  arithmetically Buchsbaum scheme of codimension $2$
such that $h^1(\mathcal{I}_Z(d-1))=1$ is the only nonzero
intermediate cohomology for $H^i_*(\mathcal{I}_Z)$ in the range $
1\leq i\leq  n-2.$ If $d\geq n-2$, then there exists an unique
one-dimensional foliation $\fol$, of degree $d$, such that
$\mathrm{Sing}(\fol)=Z$.

\end{mcor}

It is therefore worth pointing out that the latter topic appears
implicitly in works of classical algebro-geometers such as Fano,
Castelnuovo and  Palatini, in their study of varieties given by the
set of centers of complexes belonging to the linear systems of
linear complexes (see \cite{FF} and references therein).

We believe that an approach via distributions may clarify the study
of Hilbert schemes of such varieties. In fact, the Corollary
\ref{existencia fol} shows  that  there exists a one-to-one
correspondence between a class of these varieties and distribution
given by intersections of generically transversal holomorphic
distributions of codimension one.

For instance, let   $\mathcal{H}_{\ell}$ be    the union of
components of the Hilbert scheme of  $\mathbb P^n$ containing the
degeneracy locus of a general map of the form
$$
\zeta:\mathcal{O}_{\pn}(-\ell)^{\oplus (n-1)}\longrightarrow
\Omega^1_{\mathbb P^n},
$$
We have associated to $\zeta$ a globally decomposable $(n-1)$-form
$\omega_{\zeta}= \omega_1\wedge \cdots\wedge\omega_{n-1}$, where
$\omega_i\in H^0(\mathbb P^n,\Omega^1_{\mathbb P^n}(\ell)) $ for all
$i=1,\dots,n-1$. Thus, we get a natural rational map
$$
\begin{array}{ccc}
 \rho_{\ell}: \mathbf{G}(n-1,H^0(\mathbb P^n,\Omega^1_{\mathbb P^n}(\ell)) ) & \dashrightarrow & \mathcal{H}_\ell\\
 \\
 (\omega_{\zeta}=\omega_1\wedge \cdots\wedge\omega_{n-1}) & \mapsto & \mathrm{Sing}( \omega_{\zeta}) \\
\end{array}
$$
where $\mathbf{G}(n-1,H^0(\mathbb P^n,\Omega^1_{\mathbb P^n}(\ell))
)$ is the Grassmann variety   parametrizing $(n-1)$-dimensional
vector subspaces of $H^0(\mathbb P^n,\Omega^1_{\mathbb P^n}(\ell)).$
It follows form Corollary \ref{existencia fol} that $\rho_\ell$ is
generically injective. In \cite[Theorem, pg. 2 ]{FF} is showed the
birationality of $\rho$ in the case $\ell=2$.

\begin{mcor}
The map $\rho_\ell$ is generically injective, for all $\ell\geq 2$.
\end{mcor}

In the section $5$, we use the  part (ii) of Theorem \ref{3}   and
M. C. Chang's characterization of codimension two arithmetically
Buchsbaum to classify complete intersection foliations by curves
whose singular locus is smooth and non-general type.

 Finally, we give some characterizations for the splitting type for
 distributions with (co)tangent sheaf locally free and  globally generated.
\begin{mthm}\label{1}
Let $\fol$ be a distribution on $\PP^n$ of dimension $r$, degree
$d$, and such that the tangent sheaf $\F$ is locally free. Then the
following hold:
\begin{itemize}
\item[(i)] if $\F^{*}$ is globally generated and
ample, then $\singf$ is nonempty with pure dimension $r-1$.   This
holds in particular if
$\F=\oplus_{i=1}^{r}\mathcal{O}_{\PP^n}(-d_i)$ with $d_i>0$ for all
$i$. Moreover
$$
\deg(\singf)= \deg(c_{n-r+1}(\mathcal{T}_{\mathbb{P}^n}-\F))\leq
d^{n-r+1}+d^{n-r}+ \cdots +d+1.
$$
 \item [(ii)] if $\fol$ is a foliation, then $\F=\mathcal{O}_{\PP^n}(1-d)\oplus
\mathcal{O}_{\PP^n}(1)^{\oplus r-1}$ if and only if there exists a
generic linear projection $\pi: \PP^n\rightarrow\PP^{n-r+1} $ and a
one-dimensional foliation $\G$ on $\PP^{n-r+1}$, of degree $d$, with
isolated singularities, such that $\fol=\pi^*\G$.
 \item [(iii)] $\F=\mathcal{O}_{\mathbb{P}^n}(r-d)\oplus\mathcal{O}_{\mathbb{P}^n}^{\oplus
r-1}$ if and only if  $\F^*$ is globally generated and $\fol$ admits
a locally free subdistribution of rank and degree $r-1$. The last
condition can be dropped if $d=r+1$, and replaced by $c_r(\F)=0$ if
$d=r$;
\item[(iv)] if $d=r-2$ and $\F$ is globally generated and does not split, then  $\F$ is the tangent sheaf of the contact distribution
on $\mathbb{P}^3$.
\end{itemize}
\end{mthm}

\bigskip

\noindent {\bf Acknowledgements.} The present paper grew out of a
conversation between the second named author and Jorge Vitorio
Pereira regarding the paper \cite{GP}. MJ is partially supported by
the CNPq grant number 302477/2010-1 and the FAPESP grant number
2011/01071-3.

%%%%%%%%%%%%%%%%%%%%%%%%%%%%%%%%%%%%%%%%%%%%%%%%%%%%%%%%%%%%%%%%%%%%
%%%%%%%%%%%%%%%%%%%%%%%%%%%%%%%%%%%%%%%%%%%%%%%%%%%%%%%%%%%%%%%%%%%%%

\section{Preliminaries}

We start by collecting relevant definitions and results from the
literature.

%%%%%%%%%%%%%%%%%%%%%%%%

\subsection{The Eagon-Northcott resolution}\label{EN}

Let  $\mathcal{E}$  and $\mathcal{G}$ be locally free sheaves on $X$
of rank $e$ and $g$, respectively, and $\varphi:\mathcal{E}\to
\mathcal{G}$ a generically surjective morphism. The induced map
$\wedge^g\varphi: \bigwedge^g\mathcal{E} \to \det(\mathcal{G})$
corresponds to a global section  $\omega_{\varphi}\in H^0\big(X,
\bigwedge^g(\mathcal{E}^*)\otimes\det(\mathcal{G})\big)$.

\begin{defi}\label{def:deg}
The \emph{degeneracy scheme} $\mathrm{Sing}(\varphi)$ of the map
$\varphi:\mathcal{E}\to \mathcal{G}$ is the zero scheme of the
associated global section $\omega_{\varphi}\in
H^0\big(X,\bigwedge^g(\mathcal{E}^*)\otimes\det(\mathcal{G})\big)$.

Alternatively, $\omega_{\varphi}$ may also be regarded as a map
$\bigwedge^g\mathcal{E}\otimes\det(\mathcal{G})^* \to \O_X$; its
image is the ideal sheaf of $\mathrm{Sing}(\varphi)$.
\end{defi}

Suppose that $Z=\mathrm{Sing}(\varphi)$ has pure expected dimension,
i.e., $Z$ has pure codimension equal to $e-g+1$. Then the structure
sheaf of $Z$ admits a special resolution, called the
\emph{Eagon-Northcott resolution} (see for instance
\cite[A2.6]{Eisenbud}):

$$
0 \to \bigwedge^e \mathcal{E}\otimes S_{e-g}(\mathcal{G}^*) \otimes
\det(\mathcal{G}^*)\to \bigwedge^{e-1} \mathcal{E}\otimes
S_{e-g-1}(\mathcal{G}^*) \otimes \det(\mathcal{G}^*) \to \ldots
$$
$$
\to \bigwedge^{g+1}\mathcal{ E}\otimes \mathcal{G}^* \otimes
\det(\mathcal{G}^*) \to \bigwedge^{g} \mathcal{E} \otimes
\det(\mathcal{G}^*)  \to \mathcal{I}_{Z}\to 0.
$$

\subsection{Holomorphic Distributions and Pfaff systems }

For the remainder, $X$ is an irreducible, non-singular complex
scheme of dimension $n$. Its tangent sheaf is denoted by
$\mathcal{T}_X$, $\Omega^1_X$ denotes its sheaf of differentials,
and $\omega_X$ is its canonical bundle.

A \emph{(saturated) distribution $\fol$ of codimension $k$} on $X$
is a short exact sequence
\begin{equation} \label{equbas}
0 \lra \F \stackrel{\varphi}\lra \tx \stackrel{\varpi}\lra \nf \lra
0
\end{equation}
such that $\nf$, called the \emph{normal sheaf of} $\fol$, is a
nontrivial torsion free sheaf of rank $k$ on $X$. It follows that
the sheaf $\F$, called the \emph{tangent sheaf of} $\fol$, is a
reflexive sheaf. We may also refer to the \emph{dimension} of
$\fol$, which is defined as ${\rm rank}(\F)=n-k$.

A distribution $\fol$ is said to be \emph{locally free} if $\F$ is a
locally free sheaf, and one says that $\fol$ is a \emph{foliation}
if it is integrable, that is, the stalks of $\F$ are invariant under
Lie Bracket.

A dual perspective can have been considered. Indeed, a \emph{Pfaff
system $\mathcal{S}$ of codimension $k$} on $X$ is a short exact
sequence
\begin{equation} \label{equbaf}
0  \longrightarrow P_{\mathcal{S}} \stackrel{\phi}\longrightarrow
\Omega_X^1 \stackrel{\pi}\longrightarrow \Omega_{\mathcal{S}}
\longrightarrow  0
\end{equation}
such that $\Omega_{\mathcal{S}}$ is a nontrivial torsion free sheaf
of rank $n-k$ on $X$. It follows that $P_{\mathcal{S}}$ is a
reflexive sheaf of rank $k$.

Indeed, every distribution $\fol$ induces a Pfaff system
$\mathcal{S}(\fol)$ by duality as follows. Dualizing the sequence
(\ref{equbas}) one obtains the exact sequence
$$ 0 \to \nf^* \to \Omega^1_X \to \F^* \to \mathscr{E}xt^1(\nf,\mathcal{O}_{X}) \to 0 .$$
Breaking it into two short exact sequences we obtain
\begin{equation} \label{break1}
0  \longrightarrow \nf^* \longrightarrow \Omega_X^1 \longrightarrow
\Omega \longrightarrow  0
\end{equation}
and
\begin{equation} \label{break2}
0  \longrightarrow \Omega \longrightarrow \F^* \to
\mathscr{E}xt^1(\nf,\mathcal{O}_{X}) \to 0 .
\end{equation}
It is easy to see from sequence (\ref{break2}) that $\Omega$ is a
torsion free sheaf, thus sequence (\ref{break1}) is a Pfaff system.
We say that $\E:=\nf^*$ is the \emph{Pffaf sheaf} of the
distribution $\fol$. We say that the Pfaff system is \emph{ locally
free } if  $\E$ is locally free.

A third, closely related concept is that of a Pfaff field, cf.
\cite{EK}. A \emph{Pfaff field $\eta$ of rank $r$} is a nonzero map
$\eta:\Omega_X^r\to\mathcal{L}$; alternatively, since ${\rm
Hom}(\Omega_X^r,\mathcal{L})\simeq
H^0(\Omega^{n-r}_X\otimes\mathcal{L}\otimes\omega_X)$, $\eta$ can
also be regarded as a $(n-r)$-form $\omega_\eta$ with values on the
line bundle $\mathcal{L}\otimes\omega_X$.

The \emph{singular scheme ${\rm Sing}(\eta)$ of a Pfaff field
$\eta$} is the zero locus of $\omega_\eta$.
 Alternatively, $\omega_\eta$ may also be regarded as a map $\psi_\eta:\Omega^{n-r}_X\otimes\mathcal{L}^*\otimes\omega_X^*\to\mathcal{O}_X$,
 and ${\rm Sing}(\eta)$ is the subscheme of $X$ whose ideal is precisely the sheaf ${\rm im}(\psi_\eta)$.

Every distribution $\fol$ of dimension $r$ induces a Pfaff field
$\eta_{\fol}$ of rank $r$ in the following manner. Dualizing and
taking determinants of the map $\varphi:\F\to\tx$, we obtain a map:
\begin{equation}\label{eta_f}
\eta_{\fol}:\Lambda^r(\tx^*)\simeq\Omega^r_X\to\det(\F^*) .
\end{equation}
The \emph{singular scheme ${\rm Sing}(\fol)$ of a foliation $\fol$}
is defined to be the singular scheme of the corresponding Pfaff
field $\eta_{\fol}$.

Note that if $\F$ is locally free, ${\rm Sing}(\fol)$ coincides with
the degeneracy locus
$$ \Delta_\varphi := \{ x\in X ~|~ \varphi(x) ~~{\rm not~~injective} \} $$
of the map $\varphi:\F\to\tx$, thus ${\rm Sing}(\fol)$ is a
determinantal scheme. In this case, one also sees that ${\rm
Sing}(\fol)$ coincides with the singular locus of its normal sheaf,
i.e.
$$ {\rm Sing}(\fol) = {\rm Sing}(\nf) := \{ x\in X ~|~ (\nf)_x {\rm is~~not~~free}~~ \mathcal{O}_x{\rm -module} \}.$$

Every distribution $\fol$ of codimension $k$ can also be associated
to a $k$-form $\omega_{\fol}\in
H^0(\Omega^k_X\otimes\omega_X\otimes\det\F^*)$, which is the form
associated with the Pfaff field $\eta_{\fol}$. Note that
$\omega_{\fol}$ is a \emph{locally decomposable $k$-form} with
coefficients in $\omega_X\otimes\det\F^*$, i.e. for every point
$x\in X$, there are $1$-forms $\omega_1,\dots,\omega_k$ defined on
an open neighbourhood $U$ of $x$ such that
$\omega_{\fol}|_{U}=\omega_1\wedge\dots\wedge\omega_k$.

\bigskip

The sheaf $\F$ can be recovered as the kernel of the morphism
$$
\tx \lra \Omega_X^{k-1}\otimes \det \nf
$$
and we have the following diagram
$$
\xymatrix{
0\ar[r] &\F\ar[r]  & \tx\ar[d] \ar[r]& \Omega_{X}^{k-1}\otimes \det \nf & \\
&0 \ar[r]  & \nf \ar[d] \ar[ur] &&
\\
&&  0 && }
$$

\bigskip

Now, let  $\fol$ be a distribution. If the   Pfaff System $\phi:\E
\to\Omega^1_X$ induced by $\fol$ is locally free, then ${\rm
Sing}(\fol)$ coincides with the degeneracy locus
$$ \Delta_{\phi} := \{ x\in X ~|~ \phi(x) ~~{\rm not}~~{\rm injective} \} $$
of the map $\phi:\E \to\Omega^1_X$.

\bigskip

Although different from the way we opted to set up concepts here,
the reader can check
 for instance \cite[Sec. 3]{EK}, a very helpful source of how capturing the essence of this
 subject and put it within a general framework.

%Observe that  $\det(\mathcal N_{\F})=\I_Z \otimes K^*X \otimes \det(\mathcal F)$.

%%%%%%%%%%%%%%%%%%%%%%%%%%%%%%%%%%%%%%%%%%%%

%\section{Distributions on Projective Spaces}

%%%%%%%%%%%%%%%%%%%%%%%%%%%%%%%%%%%%%%%%%%%%

%\subsection{On the singular schemes of locally free distributions}

\begin{prop}\label{Sing local decomponivel}
Let $\mathcal{L}$ be a line bundle and $\omega\in
H^0(\Omega_{X}^{k}\otimes\mathcal{L})$ be a locally decomposable
section. Then $\cod(\mathrm{Sing}(\omega))\leq k+1$. In particular,
if $\fol$ is a distribution on $X$ of codimension $k$, then
$\cod(\mathrm{Sing}(\fol))\leq k+1$. Moreover, if $\nf$ is a $k$-th
syzygy sheaf, then $\cod(\singf)=k+1$.
\end{prop}

\begin{proof}
Since $\omega$ is locally decomposable, for each $p\in X$ there
exist germs of polynomial $1$-forms $\omega_1,\ldots,\omega_k$ on a
neighborhood $U$ of $p$, such that $\omega|_{U}=\omega_1
\wedge\ldots\wedge\omega_k$ and $\mathrm{Sing}(\omega)\cap
U=\{\omega_1\wedge\ldots\wedge\omega_k=0\}$. Therefore
$\mathrm{Sing}(\omega)\cap U$ is determinantal and has codimension
at most $k+1$.

Since a distribution $\fol$  of codimension $k$ on $X$ induces a
locally decomposable section $\omega\in
H^0(\Omega_{X}^{k}\otimes\det \nf)$ with
$\singf=\mathrm{Sing}(\omega)$, it immediately follows that
$\cod(\singf)\leq k+1$. On the other hand, if $\nf$ is a $k$-th
syzygy sheaf then $\cod(\singf)\geq k+1$ owing to \cite[Thm. 1.1.6,
p. 145]{Okonek}.
\end{proof}

\subsection{ACM and Arithmetically Buchsbaum schemes}
\label{acmsection}

A closed subscheme $Y\subset\pn$ is {\em arithmetically
Cohen--Macaulay (ACM)} if its homogeneous coordinate ring $S(Y) =
k[x_0,\dots,x_n]/I(Y)$ is a Cohen--Macaulay ring.

Equivalently, $Y$ is ACM if $H_*^p({\mathcal O}_Y)=0$ for $1\leq
p\leq\dim\,Y-1$ and $H_*^1({\mathcal I}_Y)=0$ (cf. \cite{CaH}). From
the long exact sequence of cohomology associated to the short exact
sequence
$$
0 \lra {\mathcal I}_Y \lra \mathcal{O}_{\pn} \lra \mathcal{O}_Y \lra
0
$$
one also deduces that $Y$ is ACM if and only if $H_*^p({\mathcal
I}_Y)=0$ for $1\leq p\leq\dim\,Y$.

Note that every ACM closed subscheme in $\pn$ is Cohen--Macaulay,
but the converse is not true.

Similarly, a closed subscheme in $\pn$ is {\em arithmetically
Buchsbaum} if its homogeneous coordinate ring is a Buchsbaum ring.
Clearly, every ACM scheme is arithmetically Buchsbaum, but the
converse is not true: the union of two disjoint lines is
arithmetically Buchsbaum, but not ACM.

We will use the following cohomological characterization of
arithmetically Buchsbaum schemes \cite{SV}, see also \cite{Ch0}.

\begin{prop}[St\"uckrad, Vogel] \label{arith Buchsbaum}
If $Y\subset\pn$ is closed subscheme such that
\begin{itemize}
\item[(i)] the multiplication map
$H^p(\I_Y(i))\stackrel{x}{\rightarrow}H^p(\I_Y(i+1))$ is zero for
every section $x\in H^0(\opn(1))$, $i\in\mathbb{Z}$ and $1\le p\le
\dim Y$;
\item[(ii)] $h^p(\I_Y(i)),h^q(\I_Y(j))\ne0$ for $1\le p<q\le \dim Y$, implies
$(p+i)-(q+j)\ne 1$;
\end{itemize}
then $Y$ is arithmetically Buchsbaum.
\end{prop}

For two-codimensional subschemes, a more precise result is found in
\cite[p. 324]{Ch}.

\begin{teo}[Chang] \label {aB codim 2}
If $Y\subset\pn$ ($n\geq 3$) is a closed subscheme of codimension
$2$, then $Y$ is arithmetically Buchsbaum if and only if the ideal
sheaf $\I_Y$ admits a resolution of the form
\begin{equation}
\label{omega res} 0 \to \oplus_i
\opn(-a_i)\to(\oplus_j\Omega^{p_j}(-k_j)^{\oplus
l_j})\oplus(\oplus_s\opn(-c_s)) \to \I_Y \to 0,
\end{equation}
where $h^{p_j}(\mathcal{I}_Z(k_j))=l_j$ are  the  only nonzero
intermediate cohomologies  for $1\leq p_j \leq n-2$.
\end{teo}

%%%%%%%%%%%%%%%%%%%%%%%%%%%%%%%%%%%%%%%%%%%%%%%%%%%%%%%%%%%%%%%%%%%%%%%%%%%%%%%%%

\subsection{Splitting criteria for locally free sheaves on $\pn$}\label{criteria}

The well-known Horrocks splitting criterion says that a locally free
sheaf $\mathcal{F}$ on $\pn$ is a sum of line bundles if and only if
it has no intermediate cohomology, i.e., $H^p(\mathcal{F}(i))=0$ for
every $1\le p\le n-1$ and every $i\in\mathbb{Z}$.

However, there are stronger splitting criteria, due to G. Evans and
P. Griffith \cite{EG} and N. M. Kumar, C. Peterson and A. P. Rao
\cite[Thm 1]{KPR} which will be relevant here. For the convenience
of the reader, let us briefly revise them.

We say that a sheaf $\mathcal{F}$ on $X$ {\em splits} if it is a sum
of line bundles.

\begin{teo}[Evans, Griffith]
\label{EG} Let $\mathcal{F}$ be a locally free sheaf on $\pn$ of
rank $r \leq n$. Then $\mathcal{F}$ splits if and only if
$H_*^p(\mathcal{F}) = 0$ for $1 \le p \le r - 1$.
\end{teo}

\begin{teo}[Kumar, Peterson, Rao]
\label{KPR} Let $\mathcal{F}$ be a locally free sheaf on $\pn$ of
rank $r$. Then
\begin{itemize}
\item[(i)] if $n$ is even and $r\le n-1$, then $\mathcal{F}$  splits iff $H_*^p(\mathcal{F}) = 0$ for $2 \le p \le n-2$;
\item[(ii)] if $n$ is odd and $r\le n-2$, then $\mathcal{F}$ splits iff $H_*^p(\mathcal{F}) = 0$ for $2 \le p \le n-2$.
\end{itemize}
\end{teo}

%%%%%%%%%%%%%%%%%%%%%%%%%%%%%%%%%%%%%%%%%%%%%%%%%%%%%%%%%%%%%%%%%%%%%%%%%%%%%%5

\subsection{Holomorphic distributions on projective spaces}

Our main concern in this work is about the holomorphic distributions
on the projective space $\pn$. In order to proceed, we need a better
descripton of them. So let $\fol$ be a codimension $k$ distribution
on $\mathbb P^n$ given by $\omega \in H^0(\mathbb
P^n,\Omega^k_{\mathbb P^n} \otimes \mathcal L)$.

If $i: \mathbb P^k \to \mathbb P^n$ is a general linear immersion
then $i^* \omega \in H^0(\mathbb P^k, \Omega^k_{\mathbb P^k} \otimes
\mathcal L)$ is a section of a line bundle, and its zero divisor
reflects the tangencies between $\fol$ and $i(\mathbb P^k)$. The
\emph{degree} of $\fol$ is, by definition, the degree of such
tangency divisor. Set $d:=\deg(\fol)$.

Since $\Omega^k_{\mathbb P^k}\otimes \mathcal L = \mathcal
O_{\mathbb P^k}( \deg(\mathcal L) - k - 1)$, one concludes that
$\mathcal L= \mathcal O_{\mathbb P^n}(d+ k + 1)$. Observe also that
$\bigwedge^k\nf=\I_Z \otimes \Oc_{\PP^n}(d+k+1)$ where $Z:=\singf$.
Besides, the Euler sequence implies that a section $\omega$ of
$\Omega^k_{\mathbb P^n} ( d + k + 1  )$ can be thought as a
polynomial $k$-form on $\C^{n+1}$ with  homogeneous coefficients of
degree $d + 1$, which we will still denote by $\omega$, satisfying
\begin{equation}
\label{equirw} i_R  \omega = 0
\end{equation}
where
$$
R=x_0 \frac{\partial}{\partial x_0} + \cdots + x_n
\frac{\partial}{\partial x_n}
$$
is the radial vector field. Thus the study of distributions of
degree $d$ on $\mathbb P^n$ reduces to the study of locally
decomposable homogeneous $k$-forms of degree $d+1$ on $\mathbb
C^{n+1}$ satisfying the relation (\ref{equirw}).

%%%%%%%%%%%%%%%%%%%%%%%%%%%%%%%%%%%%%%%%%%%%%%%%%%%%%%%%%%%%%%%%%%%%
%%%%%%%%%%%%%%%%%%%%%%%%%%%%%%%%%%%%%%%%%%%%%%%%%%%%%%%%%%%%%%%%%%%%%

\section{split distributions on projetive spaces}\label{split-dist}

In this Section we prove Theorem \ref{2}, which generalizes a result
due to Giraldo and Pan-Collantes \cite{GP} in the case of
$3$-dimensional projective space.

Before moving to the main result of this section, we apply the
splitting criteria of Section \ref{criteria}, we obtain the
following characterization of split distributions in terms of its
normal sheaf.

\begin{prop}
Let $\fol$ be a locally free distribution of codimension $k$ on
$\PP^n$.
\begin{itemize}
\item[(i)] If $\F$ splits, then $H^p_*(\nf)=0$ for $1\leq p \leq n-2$.
\item[(ii)] If $n$ is even and $H^p_*(\nf)=0$ for $1\leq p \leq n-2$, then $\F$ splits.
\item[(iii)] If $n$ is odd, $k\ge2$ and $H^p_*(\nf)=0$ for $1\leq p \leq n-2$, then $\F$ splits.
\item[(iv)] If $n$ is odd, $k=1$ and $H^1_*(\F)=H^p_*(\nf)=0$ for $1\leq p \leq n-2$, then $\F$ splits.
\end{itemize} \end{prop}

\begin{proof}
Suppose that $\F$ splits, then $H^{p}(\F(q))=0$ for $1\leq p\leq
n-1$ and all $q$. Consider, for each $q\in\mathbb{Z}$, the  exact
sequence
$$
0 \lra \mathcal \F(q) \lra \mathcal{T}_{\pn}(q) \lra \nf(q) \lra 0
$$
from which we get
$$
\ldots\lra H^p(\tpn(q))\lra H^{p}(\nf(q))\lra
H^{p+1}(\F(q))\lra\ldots
$$
From Bott's formula, $H^{p}_*(\tpn)=0$ for $1\leq p\leq n-2$. Hence
$H^p_*(\nf)=0$ for $1\leq p \leq n-2$.

Now, suppose that $H^p_*(\nf)=0$ for $1\leq p \leq n-2$. Taking the
long exact sequence of cohomology we have
$$
\ldots \lra H^{p-1}(\mathcal{N}(q))\lra H^p(\F(q))\lra
H^p(\tpn(q))\lra\ldots
$$
we conclude that $H^p_*(\F)=0$ for $2\leq p \leq n-2$. Since
$\rk(\F)\leq n-1$, item \emph{(ii)} follows from the first part of
Theorem \ref{KPR} above.

If $k\ge2$, then $\rk(\F)\leq n-2$, thus item \emph{(iii)} follows
from the second part of Theorem \ref{KPR}.

Finally, \emph{(iv)} follows from Theorem \ref{EG}.
\end{proof}

Next, we establish the first part of Theorem \ref{2}.

\begin{teo} \label{split=>acm}
Let $\fol$ be a locally free distribution on $\PP^n$ and whose
singular locus has the expected codimension $k+1$. If $\F$ splits
then $\singf$ is ACM.
\end{teo}

\begin{proof}
Let $r$, $k$, $d$ and $Z$ be, respectively, the rank, codimension,
degree and singular set of $\fol$. Consider the Eagon--Northcott
complex associated to the morphism
$\varphi^*:\Omega_{\pn}^1\to\F^*$:
\begin{equation} \label{encpx}
0 \lra \Omega^{n}_{\PP^n} \otimes
S_{k}(\F)(r-d)\stackrel{\alpha_k}\lra\Omega^{n-1}_{\PP^n} \otimes
S_{k-1}(\F)  (r-d) \stackrel{\alpha_{k-1}}\lra \ldots
\end{equation}
$$
\ldots\lra\Omega^{r+1}_{\PP^n}\otimes \F  (r-d)\stackrel{\alpha_1}
\lra \Omega^{r}_{\PP^n}  (r-d) \stackrel{\alpha_0}\lra
\mathcal{I}_{Z}\lra 0.
$$
Twist by $\mathcal{O}_{\pn}(q)$ and break it down into the short
exact sequences
\begin{gather*}
\begin{matrix}
  0 \lra  S_{k}(\F)(q-d-k-1) \lra
\Omega^{n-1}_{\PP^n} \otimes S_{k-1}(\F)  (r-d+q) \lra
\ker\alpha_{k-2}(q)
\lra 0  \\
   \vdots          \\
   0 \longrightarrow  \ker\alpha _{i}(q)  \longrightarrow
\Omega^{n-i}_{\PP^n} \otimes S_{k-i}(\F)  (r-d+q)\longrightarrow
\ker\alpha _{k-i-1}(q)  \longrightarrow
0  \\
   \vdots          \\
   0 \longrightarrow  \ker\alpha _{0}(q)  \longrightarrow
\Omega^{r}_{\PP^n}  (r-d+q)\longrightarrow  \mathcal{I}_{Z}(q)
\longrightarrow  0.
\end{matrix}
\end{gather*}
If $\F$ splits so does $S_{k}(\F)$ and hence $H^p_*(S_k(\F))=0$ for
$1\leq p\leq n-1$ by Horrocks splitting criterion. Therefore
$$
H^p(\ker\alpha_{k-2}(q))\simeq H^p(\Omega^{n-1}_{\PP^n} \otimes
S_{k-1}(\F)(r-d+q))
$$
for $1\leq p\leq n-2$ and all $q\in\mathbb{Z}$. But
$H^p_*(\Omega^{n-1}_{\PP^n} \otimes S_{k-1}(\F))=0$ for $1\leq p\leq
n-2$ by splitting $S_{k-1}(\F)$ and applying Bott's Formula term by
term. It follows that $H^p_*(\ker\alpha_{k-2})=0$ for $1\leq p\leq
n-2$. Using this in the next sequence and proceeding with the
argument we see that $H^p_*(\ker\alpha_{k-i})=0$ for $1\leq p\leq
n-i$. In particular, $H^p_*(\ker\alpha_{0})=0$ for $1\leq p\leq
n-k=r$. Thus, from the last sequence we get
$H^p_*(\mathcal{I}_{Z})=0$ for $1\leq p\leq r-1=\dim\,Z$, that is,
$Z$ is ACM, as desired.
\end{proof}

\section{Proof of Theorem \ref{2}}\label{proof thm 2}

In order to establish the second part of Theorem \ref{2}, we now
focus on one-codimensional distributions. We start with the
following result.

\begin{lema} \label{lemhum}
If $\fol$ is a distribution on $\pn$ with $\cod(\fol)=1$, then
\begin{itemize}
\item[(i)] $ H^0(\F(p))=0$ for $p\le-2$;
\item[(ii)] $H^1(\F(p))=0$ for $p\leq -\deg(\fol)-3$.
\end{itemize}
If, in addition, $\singf$ is ACM of codimension $2$ then
\begin{itemize}
\item[(iii)] $H^q(\F(p))=0$ for $2\le q\le n-2$ and all $p$ when $n\ge 4$;
\item[(iv)] $H^{n-1}(\F(p))=0$ for all $p\ne-n-1$, and $h^{n-1}(\F(-n-1))\le1$.
\end{itemize}
\end{lema}

\begin{proof}
Set $d:=\deg(\fol)$ and $Z:=\singf$. Consider the sequence
\begin{equation} \label{equen}
0 \lra \F \lra \tpn \lra \mathcal{I}_Z(d+2) \lra 0.
\end{equation}
Twisting it by $\opn(p)$ and passing to cohomology, we first have
that
$$ 0 \lra H^0(\F(p)) \lra H^0(\tpn(p)) \lra \cdots $$
Since $H^0(\tpn(p))=0$ for $p\le-2$, we also conclude that
$H^0(\F(p))=0$ for $p\leq -2$.

We also have the sequence
$$ H^0(\mathcal{I}_Z(d+2+p)) \lra H^1(\F(p)) \lra H^1(\tpn(p)). $$
Since $H^0(\mathcal{I}_Z(l))=0$ for $l\leq-1$, we conclude that
$H^1(\F(p))=0$ for $p\leq-d-3$.

For the second part of the Lemma, if $Z$ is ACM then
$H^q(\mathcal{I}_Z(d+2+p)=0$ for $1\le q\le n-2=\dim Z$ and all $p$.
Since $H^q(\tpn(p))=0$ for $1\le q\le n-2$, one concludes from the
sequence in cohomology
$$ H^{q-1}(\mathcal{I}_Z(d+2+p)) \lra H^q(\F(p)) \lra H^q(\tpn(p)) $$
that $H^q(F(p))=0$ for $2\le q\le n-2$ and all $p$.

Moreover, one also has the sequence
$$ 0 \to H^{n-1}(\F(p)) \to H^{n-1}(\tpn(p)) \to \cdots $$
from which one obtain item (iv), since $H^{n-1}(\tpn(p))=0$ for all
$p\ne-n-1$, and $h^{n-1}(\tpn(-n-1))=1$.
\end{proof}

\subsection{Proof of even dimensional case }

Now let $\fol$ be a locally free distribution of codimension one on
an even dimensional projective space $\PP^{2m}$. If $\singf$ is ACM,
then $\F$ is a locally free sheaf  of rank $2m-1$ satisfying
$H^q(\F(p))=0$ for $2\le q\le 2m-2$ and all $p$, by item (iii) of
Lemma \ref{lemhum} above. By the Kumar--Peterson--Rao splitting
criterion for even dimensional projective spaces, Theorem \ref{KPR},
it follows that $\F$ must split.

\subsection{Proof of odd dimensional case }

For odd dimensional projective spaces the situation is more
delicate, and we will need the following technical result.

\begin{lema} \label{B}
Let $\mathcal{F}$ be a coherent sheaf on $\pn$ ($n\ge4$) such that
\begin{itemize}
\item[(i)] $H^0(F(s))=0$ for $s\le -2$;
\item[(ii)] $H^q(F(s))=0$ for $-n-2\le s\le -2$ and $2\le q\le n-2$;
\item[(iii)] $H^{n-1}(F(s))=0$ for $s=-n-2$ and $-n\le s\le -2$.
\end{itemize}
If $h^{n-1}(\mathcal{F}(-n-1))=a$, then ${\rm rank}(\mathcal{F})\ge
a\cdot n$.
\end{lema}

\begin{proof}
By a theorem of Beilinson \cite[Thm 3.1.3, p. 240]{Okonek}, for any
coherent sheaf $E$ on $\pn$ there is a spectral sequence with
$E_1$-term
$$ E_1^{p,q} = H^q(E(p))\otimes\Omega^{-p}(-p)
~~{\rm where}~~ 0\le q\le n ~~{\rm and}~~ -n\le p\le 0 $$ which
converges to the graded sheaf associated to a filtration of $E$.

Applying this to $E:=\mathcal{F}(-2)$, we get that
\begin{itemize}
\item[(i)] $E_1^{p,0}=0$ for $-n\le p\le0$;
\item[(ii)] $E_1^{p,q}=0$ for $-n\le p\le0$ and $0\le q\le n-2$;
\item[(iii)] $E_1^{p,n-1}=0$ for $p\ne 1-n$.
\end{itemize}
The only nontrivial $E_1$-terms are:
\begin{itemize}
\item[(i)] $E_1^{p,1}=H^1(\mathcal{F}(p-2))\otimes\Omega^{-p}(-p)$;
\item[(ii)] $E_1^{1-n,n-1}:=H^{n-1}(\mathcal{F}(-n-1))\otimes\Omega^{n-1}(n-1)\simeq\Omega^{n-1}(n-1)^{\oplus a}$;
\item[(iii)] $E^{p,n}=H^n(\mathcal{F}(p-2))\otimes\Omega^{-p}(-p)$.
\end{itemize}

The $E_1$-terms $E_1^{p,1}$ and $E_1^{p,n}$, $-n\le p\le 0$,
together with the $d_1$-differentials $d_1^{p,1}:E_1^{p,1}\to
E_1^{p+1,1}$ and $d_1^{p,n}:E_1^{p,n}\to E_1^{p+1,n}$, respectively,
form complexes, which we denote by $\mathcal C^\bullet$ and
$\mathcal D^\bullet$, respectively; notice that all other
$d_1$-differentials must vanish.

The $E_2$-term of this spectral sequence will have the same shape:
\begin{itemize}
\item[(i)] $E_2^{p,0}=0$ for $-n\le p\le0$;
\item[(ii)] $E_2^{p,q}=0$ for $-n\le p\le0$ and $0\le q\le n-2$;
\item[(iii)] $E_2^{p,n-1}=0$ for $p\ne 1-n$,
\end{itemize}
while the only nontrivial $E_2$-terms are:
\begin{itemize}
\item[(i)] $E_1^{p,1}={\mathcal H}^p({\mathcal C}^\bullet)$;
\item[(ii)] $E_1^{1-n,n-1}:=H^{n-1}(\mathcal{F}(-n-1))\otimes\Omega^{n-1}(n-1)\simeq\Omega^{n-1}(n-1)^{\oplus a}$;
\item[(iii)] $E^{p,n}={\mathcal H}^p({\mathcal D}^\bullet)$.
\end{itemize}

It is then easy to see that all $d_2$-differentials
$d_2^{p,q}:E_2^{p,q}\to E_2^{p+2,q-1}$ must be zero, so the spectral
sequence converges already at the $E_2$-term, i.e.
$E_{\infty}^{p,q}=E_2^{p,q}$.

It follows from Beilinson's theorem that ${\mathcal H}^{p}({\mathcal
C}^\bullet)=0$ for $p\ne-1$, ${\mathcal H}^p({\mathcal
D}^\bullet)=0$ for $p>-n$ and
$$
\oplus_p E_{\infty}^{p,-p}= \Omega^{n-1}(n-1)^{\oplus
a}\oplus{\mathcal H}^{-1}({\mathcal C}^\bullet)\oplus{\mathcal
H}^{-n}({\mathcal D}^\bullet)
$$
is the graded sheaf associated to a filtration of $\mathcal{F}(-2)$.
In particular, it follows that the rank of $\mathcal{F}$ must be at
least equal to the rank of $\Omega^{n-1}(n-1)^{\oplus a}$, as
desired.
\end{proof}

Now let $\F$ be a locally free distribution of codimension one on an
odd dimensional projective space $\PP^{2m+1}$; the case $m=1$ of
locally free distribution of codimension one on $\PP^{3}$ is proved
by Giraldo and Pan-Collantes in \cite{GP}.

Thus we set $n=2m+1\ge5$. If $\singf$ is ACM, we get from the second
part of Lemma \ref{lemhum}, that $H^q(\F(p))=0$ for all $2\le q\le
n-1$ and all $p$, except for $q=n-1$ and $p=-n-1$, in which case
$h^{n-1}(\F(-n-1))\le1$.

We are then left with two possibilities. If $h^{n-1}(\F(-n-1))=0$,
it follows that $H^q(\F(p))=0$ for all $2\le q\le n-1$ and all $p$;
applying Serre duality, we conclude that $H^{q}(\F^*(p))=0$ for
$1\leq p\leq n-2$ and all $p$, thus $\F^*$ splits by the
Evans--Griffith spliting criterion, Theorem \ref{EG}, and so does
$\F$.

The second possibility, $h^{n-1}(\F(-n-1))=1$, leads to a
contradiction: by Lemma \ref{lemhum}, we get that $\F$ satisfies all
the hypotheses of Lemma \ref{B} with $a=1$; it then follows from
that $\F$ must have rank at least $n$, which contradicts the
hypothesis that $\F$ has rank $n-1$.

This completes the proof of Theorem \ref{2} in the odd dimensional
case.

%%%%%%%%%%%%%%%%%%%%%%%%%%%%%%%%%%%%%%%%%%%%%%%%%%%%%%%%%%%%%%%%%%%%
%%%%%%%%%%%%%%%%%%%%%%%%%%%%%%%%%%%%%%%%%%%%%%%%%%%%%%%%%%%%%%%%%%%%%

\section{complete intersections holomorphic foliations} \label{ci}

Let $\fol_1,\dots,\fol_k$ be generically transversal holomorphic
distributions of codimension one on $\pn$. Then $\fol=\fol_1\cap
\dots\cap \fol_k$ is a foliation of codimension $k$  on $\pn$.
Therefore, the Pfaff bundle $\E$ of $\fol$ splits, see \cite[Cor.
4.2.7 pg. 133]{J}. In this case we say that $\fol$ is a\emph{
complete intersection}   foliation.

Assuming $\dim(\fol)=1$, we have the sequence
\begin{equation}
0 \lra \E\lra \Omega_{\pn}^1 \lra
\mathcal{I}_Z(c_1(\Omega_{\pn}^1)-c_1(\E)) \lra 0.
\end{equation}
Set $s:=c_1(\E)$. Taking determinant we get a global section of
$$
\Omega_{\pn}^{n-1}\otimes \det \E^*\simeq
\mathcal{T}_{\pn}\otimes\mathcal{O}_{\pn}(-n-1-s)).
$$
Then $s=-n-d$ and $c_1(\Omega_{\pn}^1)-c_1(\E)=-n-1+d+n=d-1.$
Therefore we obtain
\begin{equation}
\label{equpff} 0 \lra \E\lra \Omega_{\pn}^1 \lra \mathcal{I}_Z(d-1)
\lra 0.
\end{equation}

We can use  the result of \cite{AC} to show that complete
intersection foliations of dimension one is determined by their
singular schemes.

\begin{teo}
\label{determination} Let $\fol$ be an one-dimensional complete
intersection foliation on $\pn$, of degree $d$, such that
$\mathrm{cod}(\mathrm{Sing}(\fol))=2$. If $\fol'$ is a
one-dimensional foliation on $\pn$, of degree $d$, with
$\mathrm{Sing}(\fol)\subset \mathrm{Sing}(\fol')$, then
$\fol'=\fol$.
\end{teo}
\begin{proof}
 Since the Pfaff bundle splits, write $\E=\bigoplus_{i=1}^{n-1}\mathcal{O}_{\pn}(-d_i-2)$ and consider
$$
0 \lra \bigoplus_{i=1}^{n-1} \mathcal{O}_{\pn}(-d_i-2)\lra
\Omega_{\pn}^1 \lra \mathcal{I}_Z(d-1) \lra 0.
$$
We have that $\singf$ is the degeneracy scheme of the induced dual
map
$$
\mathcal{T}_{\mathbb{P}^n} \to  \bigoplus_{i=1}^{n-1}
\mathcal{O}_{\pn}(d_i+2).
$$
Taking determinant we obtain a global section $\zeta_{\fol}\in
H^0(\mathcal{T}_{\pn}(d-1) )$ inducing $\fol$, where
$d-1=\sum_{i=1}^{n-1} (d_i+2)-n-1$. By Bott's formulae,
$$
H^1\Big(
\Omega^{n-1}_{\mathbb{P}^n}\otimes\bigwedge^{n}\mathcal{T}_{\mathbb{P}^n}\otimes
\bigoplus_{i=1}^{n-1} \mathcal{O}_{\pn}(-d_i-2)\Big)=
\bigoplus_{i=1}^{n-1} H^1\Big(
\Omega^{n-1}_{\mathbb{P}^n}\otimes\bigwedge^{n}T_{\mathbb{P}^n}(-d_i-2)\Big)=0.
$$
Thus, from \cite[Thm 1.1]{AC}, if $\zeta_{\fol'}\in
H^0(\mathcal{T}_{\pn}(d-1) )$ induces a one-dimensional foliation
$\fol'$ of degree $d$ with
$\mathrm{Sing}(\fol)=\mathrm{Sing}(\zeta_{\fol})\subset
\mathrm{Sing}(\zeta_{\fol'})=\mathrm{Sing}(\fol')$, then we must
have $\zeta_{\fol'}=\lambda \zeta_{\fol}$ for some $\lambda\in
\mathbb{C}^*$, since ${\rm End}(\Omega^{n-1}_{\PP^n})\cong
\mathbb{C} $  (cf. \cite[Lem 4.8]{AC}). Hence $\fol=\fol'$ as we
wish.
\end{proof}

\begin{teo} \label{complete inters. 1}
Let $\fol$ be an one-dimensional complete intersection foliation on
$\pn$, of degree $d$, such that
$\mathrm{cod}(\mathrm{Sing}(\fol))=2$. Suppose that the induced
Pfaff system  $\E \rightarrow \Omega_{\pn}^1$ is locally free. Then
$\E$ splits if and only if $\singf=Z$ is arithmetically Buchsbaum
with $h^1(\mathcal{I}_Z(d-1))=1$ being the only nonzero intermediate
cohomology for $H^i_*(\mathcal{I}_Z)$ in the range $ 1\leq i\leq
n-2.$
\end{teo}

\begin{proof}
It follows from  Theorem  \ref{aB codim 2} that  $\singf$ is
arithmetically Buchsbaum  with $h^1(\mathcal{I}_Z(d-1))=1$ being the
only nonzero intermediate cohomology for $\mathcal{I}_Z$ if and only
if the ideal sheaf $\mathcal{I}_Z$ has a  resolution of the form
$$
0 \lra \bigoplus_{i=1}^{n-1} \mathcal{O}_{\pn}(-b_i)\lra
\Omega_{\pn}^1(1-d) \lra \mathcal{I}_Z \lra 0.
$$
By (\ref{equpff}) we have that $\E=
\bigoplus_{i=1}^{n-1}\mathcal{O}_{\pn}(-b_i).$
\end{proof}

We will use the characterization of Theorem \ref{complete inters. 1}
and the  Theorem \ref{determination}  to classify complete
intersection foliations by curves whose singular locus is smooth and
non-general type.
\begin{prop}
Let $\fol$ be an one-dimensional complete intersection foliation on
$\pn$($n\geq 4$), of degree $d$, such that
$\mathrm{cod}(\mathrm{Sing}(\fol))=2$. If $\singf$ is smooth  and of
nongeneral type, then the classification of $\fol$ can be stated as:
$$
\begin{tabular}{|c |c|c|c|}
  \hline
  % after \\: \hline or \cline{col1-col2} \cline{col3-col4} ...
 $n$   & $\deg(\fol)$ & $\E$ & $\singf$  \\ \hline
 $4$    & $2$ & $\mathcal{O}_{\PP^4}(-2)^{\oplus 3}$ & \emph{smooth projected Veronese surface} \\ \hline
 $4$    & $3$ & $\mathcal{O}_{\PP^4}(-2)^{\oplus
2}\oplus\mathcal{O}_{\PP^4}(-3)$ &  \emph{$K3$ surface of genus $7$} \\
\hline
 $5$    & $3$ & $\mathcal{O}_{\PP^5}(-2)^{\oplus 4}$ &  \emph{a scroll over a plane cubic surface} \\ \hline
 $5$    & $4$ &  $\mathcal{O}_{\PP^5}(-2)^{\oplus
3}\oplus\mathcal{O}_{\PP^5}(-3)$ & $\mathbb{P}(R_2)\cap
Bl_{\mathbb{P}^2}\mathbb{P}^8$ \\ \hline
\end{tabular}
$$
where
$$
R_2:=\bigg\{v_i\wedge v_i \wedge\tau \ \bigg|\  \ v_i\in
\mathbb{C}^6,\ \tau \in \bigwedge^2\mathbb{C}^6 \bigg\}
$$
$$
\mathbb{P}(R_2)\subset
\mathbb{P}\bigg(\bigwedge^2\mathbb{C}^6\otimes
\mathcal{O}_{\mathbb{P}^5}\bigg)=\mathbb{P}^{14}\times
\mathbb{P}^5$$ and
$$Bl_{\mathbb{P}^2} \mathbb{P}^8 \subset \mathbb{P}^{8}\times \mathbb{P}^5 \subset \mathbb{P}^{14}\times \mathbb{P}^5.$$
\end{prop}
\begin{proof}
Since $\fol$ is a  complete intersection foliation, then
$\mathrm{Sing}(\fol)=Z$ has a  resolution
$$
0 \lra \bigoplus_{i=1}^{n-1} \mathcal{O}_{\pn}(-1-d-a_i)\lra
\Omega_{\pn}^1(1-d) \lra \mathcal{I}_Z \lra 0.
$$
Now, the result follows  from Chang's classification of smooth
arithmetically Buchsbaum schemes \cite[Prp 1.4, Tables I and
IV]{Ch2} along with the fact just stated that complete intersection
foliations by curves are determined by their singular loci(Theorem
\ref{determination}).
\end{proof}

The above proposition  shows  that if $n\geq 4$ and $\singf$ is
smooth  and of nongeneral type, then  $2\leq \deg(\fol)\leq 4$ and
$n=4,5$. In particular, if $\fol$ is a complete intersection
foliation on $\mathbb{P}^3$ such that $\singf$ is nongeneral type,
then $\singf$ is not smooth.

We give the following example of a complete intersection foliation
on $\mathbb{P}^3$ whose the singular set is not smooth and
nongeneral type arithmetically Buchsbaum curve.
\begin{exe}
Let $\fol_1$ and $\fol_2$ be one-codimensional foliations on
$\mathbb{P}^3$ given, respectively, by the pencils $\{\alpha
z_0+\beta z_1=0\}$ and $\{\lambda z_3+\mu z_4=0\}$. Then $\fol_1$
and $\fol_2$ are  induced, respectively, by  the $1$-forms
$\omega_1=z_{0}dz_1-z_{1}dz_{0}$ and
$\omega_2=z_{3}dz_4-z_{4}dz_{3}$. We have that the complete
intersection $\fol=\fol_1\cap\fol_2$, of degree one,  is given by
$$
\omega=\omega_1\wedge\omega_2=z_0z_2dz_1 \wedge dz_3- z_0z_3dz_1
\wedge dz_2 -z_1z_2dz_0\wedge  dz_3+z_1z_3dz_0\wedge dz_2
$$
with $\mathrm{Sing}(\fol)=\{z_0=z_1=0\}\cup\{ z_2=z_3=0\}$. Then
$\mathrm{Sing}(\fol)$ is an  arithmetically Buchsbaum curve which is
not smooth and nongeneral type with
$\E=\mathcal{O}_{\PP^3}(-2)^{\oplus 2}$.
\end{exe}

\begin{teo} \label{complete inters. 2}
Let $\fol$ be a holomorphic  distribution of dimension $r$ on $\pn$.
Suppose that
 $\cod(\singf)=r+1$ and that the induced  Pfaff system  $\E
\rightarrow \Omega_{\pn}^1$ is locally free. Then the following
hold:
\begin{itemize}
\item[(i)] if $r=2$ and $\E$ splits, then $\singf$ is arithmetically Buchsbaum, but not arithmetically Cohen Macaulay;
\item[(ii)] if $r=3$, $\E$ splits, the $|d_i-d_j| \neq 1$, and, for $n\geq 7$, the $d_i\neq 1$ as well,
then $\singf$ is arithmetically Buchsbaum, but not arithmetically
Cohen Macaulay.
\end{itemize}
\end{teo}

\begin{proof}
Assuming the Pfaff bundle splits, write
$\E=\oplus_{i}\mathcal{O}_{\pn}(a_i)$, $c:=\sum_{i}a_i$ and set
$Z:=\singf$.

To prove (ii), the Eagon-Northcott complex obtained from
$\mathcal{T}_{\pn}\to\E^*$ is
\begin{equation}
\label{equenc} 0 \lra \bigwedge^n\mathcal{T}_{\pn}\otimes
S_{2}(\E)(c)\lra\bigwedge^{n-1}\mathcal{T}_{\pn}\otimes\E(c)\lra
\bigwedge^{n-2}\mathcal{T}_{\pn}(c)\stackrel{\alpha}\lra\mathcal{I}_{Z}\lra
0.
\end{equation}
Setting $\mathcal{S}:=S_2(\E)=\sum_{i,j}\mathcal{O}_{\pn}(a_ia_j)$
and twisting (\ref{equenc}) by $\mathcal{O}_{\pn}(q)$ we get
$$
0 \lra
\mathcal{S}(n+1+c+q)\lra\bigoplus_{i=1}^{n-2}\bigg(\bigwedge^{n-1}\mathcal{T}_{\pn}(a_i+c+q)\bigg)\lra
\bigwedge^{n-2}\mathcal{T}_{\pn}(c+q)\lra\mathcal{I}_{Z}(q)\lra 0.
$$
which breaks down into the short exact sequences
\begin{equation}
\label{equse1}
 0 \lra \mathcal{S}(n+1+c+q)\lra\bigoplus_{i=1}^{n-2}\bigg(\bigwedge^{n-1}\mathcal{T}_{\pn}(a_i+c+q)\bigg) \lra
\ker\alpha(q) \lra 0
\end{equation}
\begin{equation}
\label{equse2}
   0 \longrightarrow  \ker\alpha(q)  \longrightarrow
\bigwedge^{n-2}\mathcal{T}_{\pn}(c+q)\lra\mathcal{I}_{Z}(q)
\longrightarrow  0.
\end{equation}
Now (ref{equse1}) yields
$$
H^1(\ker\alpha(q))\cong
\bigoplus_{i=1}^{n-2}H^1\bigg(\bigwedge^{n-1}\mathcal{T}_{\pn}(a_i+c+q)\bigg)\
\ \text{for}\ q\in\mathbb{Z}
$$
\begin{equation}
\label{equcd2} H^p(\ker\alpha(q))\cong
\bigoplus_{i=1}^{n-2}H^p\bigg(\bigwedge^{n-1}\mathcal{T}_{\pn}(a_i+c+q)\bigg)=0\
\ \text{for}\ q\in\mathbb{Z},\ 2\leq p\leq n-2
\end{equation}
while from (\ref{equse2}) we get
$$
H^p(\ker\alpha(q))  \lra
H^p\bigg(\bigwedge^{n-2}\mathcal{T}_{\pn}(c+q)\bigg)\lra
H^p(\mathcal{I}_{Z}(q))\lra
$$
$$
\lra  H^{p+1}(\ker\alpha(q))  \lra
H^{p+1}\bigg(\bigwedge^{n-2}\mathcal{T}_{\pn}(c+q)\bigg).
$$
Now, for $p=1$ or $3\leq p\leq n-3=\dim\, Z$, we have that
$H^p(\wedge^{n-2}\mathcal{T}_{\pn}(c+q))$ vanishes for every
$q\in\mathbb{Z}$ and so does $H^{p+1}(\ker\alpha(p))$ from
(\ref{equcd2}). It follows that
\begin{equation}
\label{equab1} H^p(\mathcal{I}_Z(q))=0\ \ \text{for}\
q\in\mathbb{Z}\ \text{and}\ p=1\ \text{or}\ 3\leq p\leq n-3
\end{equation}
From (\ref{equcd2}) again, $H^3(\ker\alpha(q))=0$ because $n\geq 5$.
Therefore
$$
H^2\bigg(\bigwedge^{n-2}\mathcal{T}_{\pn}(c+q)\bigg)\cong
H^2(\mathcal{I}_{Z}(q))
$$
and hence
\begin{equation}
\label{equab2} H^2(\mathcal{I}_{Z}(q))=0\ \text{for}\ q\neq -c-n-1.
\end{equation}
From (\ref{equab1}) and (\ref{equab2}) one rapidly sees that $Z$ is
arithmetically Buchsbaum though not ACM.

To prove (iii),  set $\mathcal{S}:=S_2(\E)$ and
$\mathcal{S}':=S_3(\E)$. Twisting the Eagon-Northcott complex
obtained from $\mathcal{T}_{\pn}\to\E^*$ by $\mathcal{O}_{\pn}(q)$
we get
$$
0 \lra
\mathcal{S}'(n+1+c+q)\lra\bigwedge^{n-1}\mathcal{T}_{\pn}\otimes\mathcal{S}(c+q)\lra
\bigwedge^{n-2}\mathcal{T}_{\pn}\otimes\E(c+q)\lra
$$
$$
\stackrel{\alpha_q'}\lra\bigwedge^{n-3}\mathcal{T}_{\pn}(c+q)\stackrel{\alpha_q}\lra\mathcal{I}_{Z}(q)\lra
0
$$
which breaks down into the short exact sequences
\begin{equation}
\label{equse3}
 0 \lra \mathcal{S}'(n+1+c+q)\lra\bigwedge^{n-1}\mathcal{T}_{\pn}\otimes\mathcal{S}(c+q)\lra
\ker\alpha_q' \lra 0
\end{equation}
\begin{equation}
\label{equse4}
   0 \longrightarrow  \ker\alpha_q'  \longrightarrow
\bigwedge^{n-2}\mathcal{T}_{\pn}\otimes\E(c+q)\lra\ker\alpha_q
\longrightarrow  0
\end{equation}
\begin{equation}
\label{equse5}
   0 \longrightarrow  \ker\alpha_q  \longrightarrow
\bigwedge^{n-3}\mathcal{T}_{\pn}(c+q)\lra\mathcal{I}_{Z}(q)
\longrightarrow  0
\end{equation}
Now (\ref{equse5}) yields
\begin{equation}
\label{equcd3} H^p(\mathcal{I}_Z(q))\cong H^{p+1}(\ker\alpha_q)\ \
\text{for}\ q\in\mathbb{Z},\ 1\leq p\leq n-2\ \text{but}\ p=2,3
\end{equation}
and also the exact sequence
\begin{equation}
\label{equse6} 0\lra H^2(\mathcal{I}_{Z}(q))\lra H^3(\ker\alpha_q)
\lra H^3\bigg(\bigwedge^{n-3}\mathcal{T}_{\pn}(c+q)\bigg)\lra
\end{equation}
$$
\lra H^3(\mathcal{I}_{Z}(q))\lra H^4(\ker\alpha_q)
$$
while from (\ref{equse4}) we get
\begin{equation}
\label{equcd4} H^{p+1}(\ker\alpha_q)\cong H^{p+2}(\ker\alpha_q')\ \
\text{for}\ q\in\mathbb{Z},\ 0\leq p\leq n-3\ \text{but}\ p=0,1
\end{equation}
and (\ref{equse3}) leads to
\begin{equation}
\label{equcd5} H^{p}(\ker\alpha_q')=0\ \ \text{for}\
q\in\mathbb{Z},\ 3\leq p+1\leq n-1
\end{equation}
thus
\begin{equation}
\label{equcd6} H^{p}(\ker\alpha_q)=0\ \ \text{for}\ q\in\mathbb{Z},\
3\leq p\leq n-1
\end{equation}
hence from (\ref{equcd3}) we get
\begin{equation}
\label{equfn1} H^{p}(\mathcal{I}_Z(q))=0\ \ \text{for}\
q\in\mathbb{Z},\ 1\leq p\leq n-4\ \text{but}\ p=1,3.
\end{equation}
Besides, (\ref{equse4}) and (\ref{equcd3}) yields
$$
H^1(\mathcal{I}_Z(q))\cong H^{2}(\ker\alpha_q)\cong
H^2\bigg(\bigwedge^{n-2}\mathcal{T}_{\pn}\otimes\E(c+q)\bigg)
$$
so
$$
H^1(\mathcal{I}_Z(q))\cong\bigoplus_{i=1}^{n-3}H^2\bigg(\bigwedge^{n-2}
\mathcal{T}_{\pn}(a_i+c+q)\bigg)
$$
which implies
\begin{equation}
\label{equfn2} H^1(\mathcal{I}_Z(q))=0\ \text{for}\ q\in\mathbb{Z}\
\text{but}\ q=-n-1-c-a_i
\end{equation}
On the other hand, from (\ref{equse6}) and (\ref{equcd5}) we have
$$
H^3(\mathcal{I}_Z(q))\cong
H^3\bigg(\bigwedge^{n-3}\mathcal{T}_{\pn}(c+q)\bigg)
$$
so
\begin{equation}
\label{equfn3} H^3(\mathcal{I}_Z(q))=0\ \text{for}\ q\in\mathbb{Z}\
\text{but}\ q=-n-1-c.
\end{equation}
Now gather (\ref{equfn1}), (\ref{equfn2}), (\ref{equfn3}) and recall
Proposition \ref{arith Buchsbaum}. If $|a_i-a_j|\neq 1$ its first
condition is trivially satisfied. For the second, if $\dim\,Z\geq
3$, i.e., $n\geq 7$, we just need
$$
(1-n-1-c-a_i)-(3-n-1-c)=-2-a_i\neq 1,
$$
that is, $a_i\neq -3$, i.e., $d_i\neq 1$.
\end{proof}

\subsection{Proof of Theorem 2 }
Now we are able to prove Theorem 2, in the Introduction. The first
part of item (i) correspond to the  Theorem \ref{complete inters. 1}
and  the second part of (i) can be derived from \cite[Lem 1.1]{Ch2}.
Part (ii) is the Theorem \ref{determination}. Parts (ii) and (iii)
follows from Theorem \ref{complete inters. 2}.

\section{Distribution with (co)tangent sheaf  globally generated}

In this section we prove the Theorem 4.

\begin{prop}
\label{nonempty} If $\fol$ is a locally free distribution on
$\PP^n$, of dimension $r$,  such that the cotangent sheaf $\F^{*}$
is globally generated and ample, then $\singf$ is nonempty with pure
dimension $r-1$. This holds in particular if
$\F=\oplus_{i=1}^{r}\mathcal{O}_{\PP^n}(-d_i)$ with $d_i>0$ for all
$i$. Moreover
$$
\deg(\singf)= \deg(c_{n-r+1}(\mathcal{T}_{\mathbb{P}^n}-\F))\leq
d^{n-r+1}+d^{n-r}+ \cdots +d+1..
$$
\end{prop}
\begin{proof}
Since $\mathcal{ T}_{\PP^n}$ and $\F^{*}$ are globally generated and
ample  then $\mathcal{ T}_{\PP^n}\otimes \F^{*}$ is so. The
nonemptiness of $\singf$ follows from \cite{FL} and the expected
 dimension follows from Bertini type Theorem, since $\mathcal{ T}_{\PP^n}\otimes \F^{*}$ is globally generated.

On the other hand, it follows from    Kempf- Laksov theorem
\cite{KL} that
$$
c_{n-r+1}(\mathcal{T}_{\mathbb{P}^n}-\F)=[\singf]=
\displaystyle\sum_{j=1}^m [V_j],
$$
where $V_1,\ldots, V_m$  are the irreducible components of $\singf$.
The bound
$$
\deg\left(\displaystyle\sum_{j=1}^m [V_j]\right)  \leq
d^{n-r+1}+d^{n-r}+ \cdots +d+1.
$$
follows from \cite[Corolarry 4.8]{So}.
\end{proof}

\begin{prop}
\label{pullback} Let $\fol$ be  a foliation on $\PP^n$ of dimension
$r$ and degree $d$. Then  the tangent sheaf
$\F=\mathcal{O}_{\PP^n}(1-d)\oplus \mathcal{O}_{\PP^n}(1)^{\oplus
r-1}$ if and only if there exists a generic linear projection $\pi:
\PP^n\rightarrow\PP^{n-r+1} $ and a one-dimensional foliation $\G$
on $\PP^{n-r+1}$, of degree $d$, with isolated singularities, such
that $\fol=\pi^*\G$.
\end{prop}

\begin{proof}
We will use the  idea of proof of \cite[Cor. 5.1]{Cukierman-Pereira}
and \cite[Lem. 2.2]{CNLPT}. In fact,  suppose that
$\F=\mathcal{O}_{\PP^n}(1-d)\oplus \mathcal{O}_{\PP^n}(-1)^{\oplus
r-1}$ , then $\F$  is induced by an $r$-form $\omega$ that may be
written as
\[
 \omega= i_{X}i_{Z_1} \cdots i_{Z_{r-1}} i_R
 \Omega \, ,
\]
where the $X$ is a homogeneous vector field of degree $d$, the $Z_j$
are constant vector fields, $R$ is the radial vector field and
$\Omega$ is the canonical volume form of $\mathbb{C}^{n+1}$.

In suitable coordinate system $(z_0, \ldots, z_{n })$ of $\mathbb
C^{n+ 1}$ we can write
$$
Z_j =\frac{\partial}{\partial z_{n-r+1+j}}
$$
for $j=1,\dots,r-1$. It follows that the fibers of the linear
projection
$$
\pi(z_0,\ldots, z_{n}) = (z_0, \ldots, z_{n-r+1})
$$
are everywhere tangent to  the leaves of $\fol$. It follows from
\cite[Lem. 2.2]{CNLPT} that there exists a one-dimensional foliation
$\G$ such that $\F= \pi^* \G.$
\end{proof}

\begin{exe}
Let $\fol$ be  a foliation on $\PP^n$ of  dimension $r$ and degree
$d$. If tangent sheaf $\F=\mathcal{O}_{\PP^n}(1-d)\oplus
\mathcal{O}_{\PP^n}(1)^{\oplus r-1}.$ Then
$$
\deg(\singf)=d^{n-r+1}+ \cdots +d+1.
$$
This follows from proposition \ref{pullback}.
\end{exe}

A \emph{subdistribution} $\G$ of $\fol$ is a distribution whose
tangent sheaf $\G\subset\F$.

\begin{prop}
\label{global1} Let $\fol$ be a  locally free distribution on
$\PP^n$ of rank $r>2$, degree $d$,  that admits a locally free
subdistribution of rank and degree $r-1$. If $\F^*$ is globally
generated, then
$\F=\mathcal{O}_{\mathbb{P}^n}(r-d)\oplus\mathcal{O}_{\mathbb{P}^n}^{\oplus
r-1}$.
\end{prop}

\begin{proof}
Let $\G$ be the locally free subsheaf of $\F$ of rank $r-1$. We have
an exact sequence
$$
0\rightarrow\G\rightarrow\F\rightarrow
\mathcal{O}_{\mathbb{P}^n}(c_1(\F)-c_1(\G))\rightarrow0.
$$
The distribution determined by $\G$ has degree $r-1$, then
$c_1(\G)=0$. Therefore there is a non trivial morphism
$\mathcal{O}_{\mathbb{P}^n}(d-r )\hookrightarrow \F^*$. This implies
that
$$
H^0(\F^*\otimes \mathcal{O}_{\mathbb{P}^n}(r-d))\neq 0.
$$
Since $\det (\F^*)=\mathcal{O}_{\mathbb{P}^n}(d-r)$, the result
follows from \cite[Prp. 1]{Si}
\end{proof}

\begin{remark}
\label{remprt} Consider a distribution on $\PP^n$ of rank $r>1$,
degree $r+1$, and such that the cotangent sheaf $\F^*$ is globally
generated. Then $\F= \mathcal{O}_{\mathbb{P}^n}(-1)\oplus
\mathcal{O}_{\mathbb{P}^n}^{\oplus r-1}.$ In fact, it follows from
\cite[p. 53]{Okonek} that $c_1(\F^*)=1$ iff $\F^*\simeq
\mathcal{O}_{\mathbb{P}^n}(1)\oplus
\mathcal{O}_{\mathbb{P}^n}^{\oplus r-1} $.
\end{remark}

\begin{prop}
\label{global2}
 Let $\fol$ be a distribution on $\PP^n$ of rank  $r$,
degree $d$, such that $\F^*$ is globally generated. If $c_r(\F)=0$
then $\fol$ admits a locally free subdistribution of rank $r-1$ and
degree $d-1$.
\end{prop}

\begin{proof}
It follows from \cite[Lem. 4.3.2]{Okonek} that there exists a bundle
$\G^*$ of rank $r-1$ and an exact sequence $0\rightarrow
\mathcal{O}_{\mathbb{P}^n}\rightarrow\F^*\rightarrow\G^*\rightarrow
0$. In particular, $\G$ is a subbundle of $\F$ and yields a locally
free subdistribution $\G$. Besides,
$$
r-\deg(\fol)=c_1(\F)=c_1(\G)=r-1-\deg(\G)
$$
and we are done.
\end{proof}

\begin{cor}\label{global trivial}
Let $\fol$ be a distribution on $\PP^n$ of rank and degree $r$, such
that $\F^*$ is globally generated. If $c_r(\F)=0$ then
$\F=\mathcal{O}_{\mathbb{P}^n}^{\oplus r}$.
\end{cor}

\begin{proof}
Combine Propositions \ref{global1} and \ref{global2}.
\end{proof}

Recall that the  twisted null-correlation bundle $\mathscr N(1)$ on
$\pn$, with $n$ odd, is defined by the short exact sequence (cf.
\cite[p. 77]{Okonek})
$$ 0 \to \mathscr N(1) \to \tpn \to \opn(2) \to 0
$$
and it is the tangent sheaf of the contact distribution on $\PP^n$.

In \cite{Si-U} J. C. Sierra and  L. Ugaglia  proved that a  globally
generated vector bundle $\mathcal{F}$ on $\pn$ such that ${\rm
rank}(\mathcal{F})<n$ and $c_1(\mathcal{F})=2$, always splits unless
$\mathcal{F}$ is a twisted null-correlation bundle on
$\mathbb{P}^3$. Using this we can prove the following.

\begin{prop}\label{split gerado pn}
Let $\fol$ be a distribution on $\PP^n$ of rank $r>1$, degree $r-2$,
and such that $\F$ is globally generated. Then\begin{itemize}
\item [(i)]$\F=
\mathcal{O}_{\mathbb{P}^n}(1)^{\oplus
2}\oplus\mathcal{O}_{\mathbb{P}^n}^{\oplus r-2}$;
  \item [(ii)]$\F$ is the  twisted null-correlation bundle on $\PP^3$.
\end{itemize}
In particular, if $\F$ is not split, then $n=3$, $\fol$ is not a
foliation and $\singf=\emptyset$.
\end{prop}
\begin{proof}
It follows from the J. Sierra and L. Ugaglia classification
\cite{Si-U} and the fact that $c_1(\F)=2$.
\end{proof}

\subsection{Proof of Theorem \ref{1} }

Item (i) is part of Proposition \ref{nonempty} while (ii)
 corresponds to Proposition \ref{pullback}. Besides, Proposition \ref{global1}, Remark \ref{remprt} and Corollary \ref{global trivial} prove
 (iii), and (iv) is the conclusion
of Proposition \ref{split gerado pn}.

\end{document}